\documentclass[11pt,reqno]{amsart}

\usepackage{amsmath,amssymb,amsthm}
\usepackage[margin=1.25in]{geometry}
\usepackage[hidelinks]{hyperref}

\theoremstyle{plain}
\newtheorem{theorem}{Theorem}[section]
\newtheorem{lemma}[theorem]{Lemma}
\newtheorem{proposition}[theorem]{Proposition}
\newtheorem{corollary}[theorem]{Corollary}
\newtheorem{conjecture}[theorem]{Conjecture}

\theoremstyle{definition}
\newtheorem{definition}[theorem]{Definition}
\newtheorem{example}[theorem]{Example}

\theoremstyle{remark}
\newtheorem{remark}[theorem]{Remark}

\newcommand{\R}{\mathbb{R}}
\newcommand{\Z}{\mathbb{Z}}
\newcommand{\T}{\mathbb{T}}
\newcommand{\eps}{\varepsilon}
\newcommand{\abs}[1]{\left\lvert #1\right\rvert}
\newcommand{\tnorm}[1]{\left\lVert #1\right\rVert_{\T}}
\newcommand{\fhat}[1]{\widehat{#1}}
\newcommand{\dimHlog}{\dim_{\mathrm H}^{\log}}
\newcommand{\dimPlog}{\dim_{\mathrm P}^{\log}}
\newcommand{\overdimBlog}{\overline{\dim}_{\mathrm B}^{\log}}
\DeclareMathOperator{\supp}{supp}
\DeclareMathOperator{\dist}{dist}
\DeclareMathOperator{\sgn}{sgn}

\begin{document}

\raggedbottom

\title[Erd\H{o}s similarity and Rajchman measures]
{The Erd\H{o}s similarity conjecture and Rajchman measures}

\author[Iosevich]{A. Iosevich}
\address{Department of Mathematics, University of Rochester, Rochester, NY, USA}
\email{iosevich@gmail.com}

\author[Kulkarni]{N. Kulkarni}
\address{Department of Mathematics, University of Rochester, Rochester, NY, USA}
\email{nkulkar7@math.rochester.edu}

\author[Mora Cuéllar]{N. Mora Cuéllar}
\address{Department of Mathematics, University of British Columbia, Vancouver, BC, Canada}
\email{natalia.mora@math.ubc.ca}

\author[Rojas Aravena]{I. Rojas Aravena}
\address{Department of Mathematics, University of British Columbia, Vancouver, BC, Canada}
\email{i.andres@math.ubc.ca}

\author[Yavicoli]{A. Yavicoli}
\address{Department of Mathematics, University of British Columbia, Vancouver, BC, Canada}
\email{yavicoli@math.ubc.ca}

\thanks{A.~I. was supported in part by National Science Foundation grant
DMS-2154232. A.~Y. was supported in part by the Natural Sciences and Engineering Research
Council of Canada (NSERC), grants GR030571 and GR030540.}

\date{August 2026}


\begin{abstract}
Let $A\subseteq\R$ support a probability measure whose Fourier--Stieltjes transform tends
to zero at infinity. We prove that, for every $\eps\in(0,1)$, there is a closed,
$1$-periodic, nowhere dense set $E\subseteq\R$ such that
\[
m(E\cap I)\ge1-\eps
\]
for every interval $I$ of length $1$, while $E$ contains no affine copy of $A$. Thus every
set supporting a Rajchman measure satisfies the Erd\H{o}s similarity conjecture in a
uniform large-set form. The proof combines equidistribution modulo one for large dilates
of the measure with a multiscale family of low-density periodic blockers; no quantitative
rate of Fourier decay is used. We also refine a classical theorem of Iva\v{s}ev-Musatov,
showing that for
every Hausdorff gauge $h$ there is an $h$-null compact Rajchman support $K$ satisfying
\[
\overdimBlog K=\dimPlog K=1.
\]
The value $1$ is sharp for both dimensions.
\end{abstract}

\maketitle

\tableofcontents

\section{Introduction}\label{sec:intro}

Throughout, $m$ denotes Lebesgue measure on $\R$. An \emph{affine copy} of a set
$A\subseteq\R$ is a set
\[
x+sA=\{x+sa:a\in A\},
\qquad\text{where }x\in\R\text{ and }s\in\R\setminus\{0\}.
\]

\begin{definition}\label{def:universal}
A set $A\subseteq\R$ is \emph{universal} if every Lebesgue measurable set
$E\subseteq\R$ with $m(E)>0$ contains an affine copy of $A$.

A Lebesgue measurable set $E\subseteq\R$ is \emph{conull} if
$m(\R\setminus E)=0$. A set $A\subseteq\R$ is \emph{full-measure universal} if every
conull Lebesgue measurable set contains an affine copy of $A$.
\end{definition}

Every universal set is full-measure universal. Thus failure of full-measure universality is
a stronger conclusion than ordinary non-universality.

Every finite set is universal by the Lebesgue density theorem, and Erd\H{o}s asked whether
these are the only universal subsets of $\R$ \cite{Erdos}. Classical partial results for
slowly decaying sequences are due independently to Falconer and Eigen
\cite{Falconer,Eigen}, while Bourgain proved non-universality for triple sumsets of
infinite sets \cite{Bourgain}. Subsequent avoidance results and large- or full-measure
variants include work of Kolountzakis \cite{Kolountzakis,KolountzakisCantor}, Cruz, Lai, and
Pramanik \cite{CruzLaiPramanik}, Kolountzakis and Papageorgiou
\cite{KolountzakisPapageorgiou}, Gao, Mooroogen, and Yip \cite{GaoMooroogenYip}, and
Shmerkin and Yavicoli \cite{ShmerkinYavicoli}. A bi-Lipschitz variant was studied by
Feng, Lai, and Xiong \cite{FengLaiXiong}. For broader context, see the surveys
\cite{Svetic,JLM}.

For related recent results on thin lattice examples and structured two-fold sumsets,
obtained by additive mechanisms different from the one used here, see
\cite{IosevichYavicoliFalconer,MoraEtAlSumsets}.

\begin{conjecture}[Erd\H{o}s similarity conjecture]\label{conj:erdos}
No infinite subset of $\R$ is universal.
\end{conjecture}

We shall use a quantitative notion of largeness that has appeared in recent work on the
Erd\H{o}s similarity problem in the large
\cite{BradfordKohutMooroogen,GaoMooroogenYip}.

\begin{definition}\label{def:large}
Let $\eps\in(0,1)$. A Lebesgue measurable set $E\subseteq\R$ is
\emph{$(1-\eps)$-large} if
\[
m(E\cap I)\ge1-\eps
\]
for every interval $I\subseteq\R$ of length $1$.
\end{definition}

The condition is uniform over all unit intervals. In particular, it is substantially stronger
than asking only that $E$ have positive measure.

Our hypothesis is expressed in terms of a classical Fourier-analytic object.

\begin{definition}\label{def:rajchman}
For a finite Borel measure $\mu$ on $\R$, its Fourier--Stieltjes transform is
\[
\fhat\mu(\xi)=\int_{\R}e^{-2\pi i\xi x}\,d\mu(x),
\qquad\text{for all }\xi\in\R.
\]
A probability measure $\mu$ is a \emph{Rajchman measure} if
\[
\fhat\mu(\xi)\xrightarrow[\abs{\xi}\to\infty]{}0.
\]
We say that $A\subseteq\R$ \emph{supports} a Rajchman measure if there is a Rajchman
probability measure $\mu$ with $\supp\mu\subseteq A$.
\end{definition}

Our main theorem gives a uniform large-set version of non-universality.
\begin{theorem}\label{thm:main}
Let $A\subseteq\R$ support a Rajchman measure. For every $\eps\in(0,1)$ there exists a
closed, $1$-periodic, nowhere dense set $E\subseteq\R$ such that
\[
m(E\cap I)\ge1-\eps
\]
for every interval $I\subseteq\R$ of length $1$, and
\[
x+sA\not\subseteq E
\qquad
\text{for every }x\in\R\text{ and every }s\in\R\setminus\{0\}.
\]
In particular, $A$ is non-universal, so the Erd\H{o}s similarity conjecture holds for $A$.
\end{theorem}

Intersecting the set in Theorem~\ref{thm:main} with one unit interval gives the following
compact consequence.

\begin{corollary}\label{cor:compact-avoiding}
Let $A\subseteq\R$ support a Rajchman measure. For every $\eps\in(0,1)$ there exists a
compact set $E\subseteq[0,1]$ such that
\[
m(E)\ge1-\eps
\]
and
\[
x+sA\not\subseteq E
\qquad
\text{for every }x\in\R\text{ and every }s\in\R\setminus\{0\}.
\]
\end{corollary}

Only qualitative Fourier decay is used in Theorem~\ref{thm:main}; no modulus of decay is
assumed. This distinction matters. Power Fourier decay forces positive Hausdorff dimension,
whereas compact Rajchman supports may have Hausdorff dimension zero and may be much
thinner still.

Rajchman measures have a long history in harmonic analysis; see \cite{Lyons}. On the
circle, a closed set carrying a nonzero Rajchman measure is traditionally called an
$M_0$-set, while a set annihilated by every Rajchman measure is called a $U_0$-set. For
compact supports, after an affine normalization, this circle formulation is equivalent to
the real-line formulation used here; see Remark~\ref{rem:circle-line}. Two classical facts
are especially relevant. First, the class of Rajchman measures is a band: absolute
continuity with respect to a Rajchman measure preserves Fourier decay. This was proved by
Rajchman and Milicer-Gru\.zewska and is recalled in Lemma~\ref{lem:band}. Second,
Iva\v{s}ev-Musatov proved that, for every Hausdorff gauge $h$, there is a closed set of
zero $h$-Hausdorff measure which is not a $U_0$-set \cite{IvashevMusatov}; see also
\cite{Lyons}. K\"orner obtained far-reaching refinements of this theorem
\cite{KornerIvashev}. Thus the existence of arbitrarily thin Rajchman supports in the sense
of Hausdorff gauges is classical.

The contribution of our thin-set result is the simultaneous sharp control of logarithmic box and
packing size. We write
\[
\dimHlog,\qquad \dimPlog,\qquad \overdimBlog
\]
for logarithmic Hausdorff, packing, and upper box dimensions; the definitions are given in
Subsection~\ref{subsec:log-dimensions}. We prove in Theorem~\ref{thm:thin-support} that,
for every Hausdorff gauge $h$, there is a compact set $K\subseteq(0,1)$ and a Rajchman
probability measure $\mu$ such that
\[
\supp\mu=K,
\qquad
\mathcal H^h(K)=0,
\qquad
\overdimBlog K=\dimPlog K=1.
\]
The value $1$ is the smallest possible value of either of these two logarithmic dimensions for a
compact Rajchman support. Choosing a gauge which dominates every logarithmic power gives, in
particular,
\[
\dimHlog K=0,
\qquad
\dimPlog K=1.
\]
Bluhm's construction of a zero-Hausdorff-dimensional Rajchman support
\cite{Bluhm} is an important earlier example in this direction.

These endpoint examples satisfy neither dimensional hypothesis in a recent theorem of
Shmerkin and Yavicoli \cite{ShmerkinYavicoli}. In the one-dimensional affine case, their
result says that if a Borel set $A$ satisfies either
\[
\dimHlog A>1
\qquad\text{or}\qquad
\dimPlog A>2,
\]
then there is a conull measurable set containing no affine copy of $A$; see
Theorem~\ref{thm:shmerkin-yavicoli}. Their conclusion is stronger in measure whenever
their hypotheses apply. Theorem~\ref{thm:main}, on the other hand, reaches compact sets at
the sharp logarithmic packing endpoint and produces an avoiding set which is closed,
periodic, nowhere dense, and uniformly large in every unit interval.

We now describe the main ideas. The band property first allows us to restrict any
Rajchman measure to a compact set of positive mass. If $K$ is the resulting compact
support and $\tau>0$, push the measure forward under
\[
x\longmapsto \tau x\bmod1.
\]
Its $n$-th Fourier coefficient is $\fhat\mu(n\tau)$. Therefore these pushforwards converge
weak-$*$ to normalized Lebesgue measure on $\T=\R/\Z$ as $\tau$ tends to infinity. A
compactness argument then shows that, for every $\eta>0$, the set
$\tau K\bmod1$ is $\eta$-dense for every $\tau$ beyond one fixed threshold.

For an integer $q\ge1$, the periodic blocker
\[
\mathcal B(q,\eta)=\{u\in\R:\tnorm{qu}<\eta\}
\]
has measure $2\eta$ in every unit interval. If $\abs{s}q$ is beyond the density threshold,
then $\mathcal B(q,\eta)$ meets every translate $x+sK$. We choose widths $\eta_k$ with
summable total measure and frequencies $q_k$ so that the $k$-th blocker detects every
scale $\abs{s}\ge2^{-k}$. The complement of their union is closed and $1$-periodic, is
$(1-\eps)$-large, and contains no affine copy of $K$. Since the frequencies tend to
infinity, the union of the blockers is dense, and the complement is nowhere dense.

Section~\ref{sec:equi} recalls the band property and proves uniform density of Rajchman
dilates. Section~\ref{sec:blockers} gives the periodic blocker construction and proves
Theorem~\ref{thm:main}. Section~\ref{sec:examples} records examples and necessary
topological structure. Section~\ref{sec:log-size} establishes the sharp logarithmic lower
bounds and constructs endpoint supports of arbitrary Hausdorff thinness.

\section{Restrictions and equidistribution of Rajchman measures}\label{sec:equi}

\subsection{The band property}

We begin with the classical band property of Rajchman measures. It is due to Rajchman and
Milicer-Gru\.zewska; see \cite[Section~2]{Lyons}. We include the proof because the
restriction statement will be used repeatedly.

\begin{lemma}[Band property]\label{lem:band}
Let $\mu$ be a finite positive Borel measure on $\R$ such that
\[
\fhat\mu(\xi)\xrightarrow[\abs\xi\to\infty]{}0.
\]
For $g\in L^1(\mu)$, let $g\mu$ denote the finite complex Borel measure defined by
\[
(g\mu)(B)=\int_B g\,d\mu
\]
for every Borel set $B\subseteq\R$. Then
$\fhat{g\mu}(\xi)\xrightarrow[\abs\xi\to\infty]{}0$. In particular, if $F$ is Borel and
$\mu(F)>0$, then the normalized restriction
\[
\frac{\mu|_F}{\mu(F)}
\]
is a Rajchman probability measure.
\end{lemma}

\begin{proof}
We first show that the linear span of the characters $x\mapsto e^{2\pi itx}$,
$t\in\R$, is dense in $L^1(\mu)$. If the span were not dense, the Hahn--Banach theorem
and the identification $(L^1(\mu))^*=L^\infty(\mu)$ would give a nonzero
$h\in L^\infty(\mu)$ such that
\[
\int_{\R}h(x)e^{2\pi itx}\,d\mu(x)=0
\qquad\text{for every }t\in\R.
\]
The integral is $\fhat{h\mu}(-t)$, so the finite complex measure $h\mu$ would have
identically zero Fourier--Stieltjes transform. Uniqueness of Fourier transforms of finite
measures would imply $h\mu=0$, and hence $h=0$ $\mu$-almost everywhere, a
contradiction.

Now fix $g\in L^1(\mu)$ and $\delta>0$. Choose
\[
P(x)=\sum_{j=1}^{J}c_je^{2\pi it_jx}
\]
such that $\lVert g-P\rVert_{L^1(\mu)}<\delta$. Then
\[
\fhat{P\mu}(\xi)
=\sum_{j=1}^{J}c_j\fhat\mu(\xi-t_j)
\xrightarrow[\abs\xi\to\infty]{}0.
\]
Moreover,
\[
\abs{\fhat{g\mu}(\xi)-\fhat{P\mu}(\xi)}
\le\lVert g-P\rVert_{L^1(\mu)}<\delta
\qquad\text{for every }\xi\in\R.
\]
It follows that
\[
\limsup_{\abs\xi\to\infty}\abs{\fhat{g\mu}(\xi)}\le\delta.
\]
Since $\delta>0$ was arbitrary, $g\mu$ is Rajchman. The final assertion follows by
taking $g=\mathbf 1_F/\mu(F)$.
\end{proof}

\begin{remark}\label{rem:circle-line}
For compact supports, the usual circle formulation and the real-line formulation of the
Rajchman property are equivalent after an affine normalization. To explain the only
nontrivial direction, suppose that $\mu$ is supported in an interval of length less than
$1$ and that $\fhat\mu(n)$ tends to zero as $\abs n$ tends to infinity through the
integers. On this interval the integer characters separate points. The same proof as in
Lemma~\ref{lem:band}, now using integer characters and the Stone--Weierstrass theorem,
shows that $\fhat{g\mu}(n)$ tends to zero for every $g\in L^1(\mu)$. Write
$\xi=n+t$, where $n\in\Z$ and $t\in[0,1)$. The functions
\[
g_t(x)=e^{-2\pi itx},\qquad 0\le t\le1,
\]
form a compact subset of $L^1(\mu)$. A finite-net argument therefore makes the decay of
\[
\fhat{g_t\mu}(n)=\fhat\mu(n+t)
\]
uniform in $t$. Hence $\fhat\mu(\xi)$ tends to zero as $\abs\xi$ tends to infinity on the
whole real line.

Conversely, real-line decay plainly implies decay along the integer frequencies. If a
Rajchman measure on $\T$ is not already carried by an arc of length less than $1$, one
first restricts it to such an arc of positive mass and uses the circle version of the band
property. This is the precise way in which the classical $M_0$-set results quoted in the
introduction pass to our real-line setting.
\end{remark}

\subsection{Uniform density modulo one}

Write $\T=\R/\Z$, and let $\lambda$ be normalized Lebesgue measure on $\T$, obtained
by pushing forward Lebesgue measure on $[0,1)$. For $t\in\R$, put
\[
\tnorm{t}=\dist(t,\Z).
\]
This induces the quotient metric
\[
d_{\T}(x+\Z,y+\Z)=\tnorm{x-y}
\qquad\text{for all }x,y\in\R.
\]
We identify $\T$ with $[0,1)$ when convenient.
A nonempty set $S\subseteq\T$ is \emph{$\eta$-dense} if for every $y\in\T$ there exists
$z\in S$ such that $d_{\T}(y,z)\le\eta$.

Let $\mu$ be a compactly supported Rajchman probability measure on $\R$, and put
$K=\supp\mu\subseteq\R$. For $\tau>0$, define
\[
\rho_\tau:\R\longrightarrow\T,
\qquad
\rho_\tau(x)=\tau x\bmod1,
\]
and let
\[
\nu_\tau=(\rho_\tau)_*\mu
\]
be the pushforward of $\mu$. Thus, for every Borel set $B\subseteq\T$,
\[
\nu_\tau(B)
=\mu\bigl(\rho_\tau^{-1}(B)\bigr)
=\mu\bigl(\{x\in\R:\tau x\bmod1\in B\}\bigr),
\]
and, equivalently, for every $f\in C(\T)$,
\[
\int_{\T}f(y)\,d\nu_\tau(y)
=\int_{\R}f(\tau x\bmod1)\,d\mu(x).
\]
Moreover,
\[
\supp\nu_\tau=\rho_\tau(K)=\tau K\bmod1.
\]
Indeed, continuity of $\rho_\tau$ and compactness of $K$ show that $\rho_\tau(K)$ is
compact and hence closed. Since $\mu(\R\setminus K)=0$, the measure $\nu_\tau$ is
concentrated on $\rho_\tau(K)$. Conversely, if
$x\in K$ and $U\subseteq\T$ is an open neighborhood of $\rho_\tau(x)$, then
$\rho_\tau^{-1}(U)$ is an open neighborhood of $x$. Since $x\in\supp\mu$,
$\mu(\rho_\tau^{-1}(U))>0$, and therefore $\rho_\tau(x)\in\supp\nu_\tau$.

\begin{lemma}[Uniform density of dilates]\label{lem:equi}
For every $\eta>0$ there is a constant $\Lambda(\eta)\ge1$ such that
\[
\tau K\bmod1\quad\text{is $\eta$-dense in $\T$}
\]
for every real $\tau\ge\Lambda(\eta)$.
\end{lemma}

\begin{proof}
We first prove that
\[
\nu_\tau\xrightarrow[\tau\to\infty]{\text{weak-$*$}}\lambda.
\]
For $n\in\Z$, the $n$-th Fourier coefficient of $\nu_\tau$ is
\[
\fhat{\nu_\tau}(n)
=\int_{\T}e^{-2\pi iny}\,d\nu_\tau(y)
=\int_{\R}e^{-2\pi in\tau x}\,d\mu(x)
=\fhat\mu(n\tau).
\]
Here reduction modulo one does not affect the exponential because $n$ is an integer.
For $n=0$, the displayed quantity equals $1$. If $n\ne0$, then $\abs{n\tau}$ tends to
infinity with $\tau$, and the Rajchman property gives
\[
\fhat{\nu_\tau}(n)=\fhat\mu(n\tau)
\xrightarrow[\tau\to\infty]{}0.
\]
These are precisely the Fourier coefficients of $\lambda$. For a trigonometric
polynomial $P(y)=\sum_{\abs{n}\le N}c_ne^{2\pi iny}$, integration is a finite linear
combination of Fourier coefficients, and hence
\[
\int_{\T}P(y)\,d\nu_\tau(y)
=\sum_{\abs{n}\le N}c_n\fhat{\nu_\tau}(-n)
\xrightarrow[\tau\to\infty]{}c_0
=\int_{\T}P(y)\,d\lambda(y).
\]
By Fej\'er's theorem \cite[Theorem~I.3.1(b), p.~18]{Katznelson}, the Ces\`aro means of
the Fourier series of every $f\in C(\T)$ converge uniformly to $f$. Since each such mean
is a trigonometric polynomial, given $f\in C(\T)$ and $\delta>0$, choose a trigonometric
polynomial $P$ satisfying $\lVert f-P\rVert_\infty<\delta$. Since $\nu_\tau$ and
$\lambda$ are probability measures,
\[
\begin{aligned}
\left|
\int_{\T}f(y)\,d\nu_\tau(y)
-\int_{\T}f(y)\,d\lambda(y)
\right|
&\le
\int_{\T}\abs{f(y)-P(y)}\,d\nu_\tau(y)\\
&\quad+
\left|
\int_{\T}P(y)\,d\nu_\tau(y)
-\int_{\T}P(y)\,d\lambda(y)
\right|\\
&\quad+
\int_{\T}\abs{P(y)-f(y)}\,d\lambda(y)\\
&\le
2\delta+
\left|
\int_{\T}P(y)\,d\nu_\tau(y)
-\int_{\T}P(y)\,d\lambda(y)
\right|.
\end{aligned}
\]
Letting first $\tau$ tend to infinity and then $\delta$ tend to zero proves the asserted
weak-$*$ convergence.

We now convert weak-$*$ convergence into the asserted density statement, with one
threshold valid for every sufficiently large $\tau$. Suppose to the contrary that the
conclusion fails for some $\eta>0$. Then, for every $j\ge1$, there are
$\tau_j\ge j$ and $y_j\in\T$ such that
\[
d_{\T}(y_j,z)>\eta
\qquad
\text{for every }z\in\tau_jK\bmod1.
\]
After passing to a subsequence, compactness of $\T$ gives a limit $y\in\T$ such that
$y_j$ tends to $y$ as $j$ tends to infinity.
For all sufficiently large $j$, $d_{\T}(y_j,y)<\eta/2$. The open ball
$B_{\T}(y,\eta/2)$ is then disjoint from
$\tau_jK\bmod1=\supp\nu_{\tau_j}$. Choose a nonzero $f\in C(\T)$ such that
$f\ge0$ and
\[
\supp f\subseteq B_{\T}(y,\eta/2).
\]
It follows that
\[
\int_{\T}f\,d\nu_{\tau_j}=0
\qquad\text{for all sufficiently large }j,
\]
whereas $\int_{\T}f\,d\lambda>0$. This contradicts the weak-$*$ convergence just proved,
since $\tau_j$ tends to infinity with $j$. Therefore a finite
$\Lambda(\eta)\ge1$ exists and works for every $\tau\ge\Lambda(\eta)$.
\end{proof}

\section{Periodic blockers and large avoiding sets}\label{sec:blockers}

For an integer $q\ge1$ and $0<\eta<\tfrac12$, define the open periodic blocker
\[
\mathcal B(q,\eta)
=\left\{u\in\R:\tnorm{qu}<\eta\right\}.
\]
Let $T_q:[0,1)\longrightarrow\T$ be given by $T_q(u)=qu\bmod1$. Then
\[
(T_q)_*\bigl(m|_{[0,1)}\bigr)=\lambda.
\] 
Indeed, partition $[0,1)$ into the intervals $[j/q,(j+1)/q)$, $0\le j<q$.
For every Borel set $F\subseteq\T$, the part of $T_q^{-1}(F)$ in each of these $q$
intervals has measure $\lambda(F)/q$, because the corresponding branch of $T_q$ is
affine with slope $q$. Hence
\[
m\bigl(T_q^{-1}(F)\bigr)=\lambda(F).
\]
Taking $F=\{t\in\T:\tnorm{t}<\eta\}$ gives
\begin{equation}\label{eq:blocker-measure}
m\bigl(\mathcal B(q,\eta)\cap[0,1]\bigr)=2\eta.
\end{equation}
Since $\mathcal B(q,\eta)$ is $1$-periodic, the same identity holds with $[0,1]$
replaced by any interval of length $1$.

The next lemma explains how density modulo one forces an affine copy to meet a blocker.

\begin{lemma}[Single-scale blocking]\label{lem:block}
Let $K\subseteq\R$ be nonempty, let $q\ge1$ be an integer, let
$0<\eta<\tfrac12$, and let $s\in\R\setminus\{0\}$. If
\[
\abs{s}qK\bmod1
\quad\text{is $\eta/2$-dense in $\T$,}
\]
then
\[
(x+sK)\cap\mathcal B(q,\eta)\ne\emptyset
\qquad
\text{for every }x\in\R.
\]
\end{lemma}

\begin{proof}
Fix \(x\in\R\), and let \(\sigma=\sgn(s)\). By the assumed
\(\eta/2\)-density, applied to the target
\(-\sigma qx\bmod 1\), there exists \(k\in K\) such that
\[
\bigl\||s|qk+\sigma qx\bigr\|_{\mathbb T}
\le \frac{\eta}{2}.
\]
Since \(s=\sigma|s|\) and multiplication by \(\sigma\in\{-1,1\}\)
does not change distance to \(\Z\),
\[
\begin{aligned}
\bigl\|q(x+sk)\bigr\|_{\mathbb T}
&=\bigl\|qx+\sigma|s|qk\bigr\|_{\mathbb T}\\
&=\bigl\|\sigma qx+|s|qk\bigr\|_{\mathbb T}\\
&\le \frac{\eta}{2}<\eta.
\end{aligned}
\]
Thus \(x+sk\in\mathcal B(q,\eta)\), and hence
\[
(x+sK)\cap\mathcal B(q,\eta)\ne\emptyset .
\]
\end{proof}

We now carry out the multiscale construction. The proposition is stated in a form which
separates the geometric blocking argument from the Fourier-analytic input.

\begin{proposition}[Multiscale blocking]\label{prop:multiscale}
Let $K\subseteq\R$ be a nonempty compact set with the following property: for every
$\delta>0$ there is a finite $\Lambda(\delta)\ge1$ such that
\[
\tau K\bmod1\quad\text{is $\delta$-dense in $\T$}
\]
for every $\tau\ge\Lambda(\delta)$. Then for every $\eps\in(0,1)$ there exists a
closed, $1$-periodic, nowhere dense set $E\subseteq\R$ such that
\[
m(E\cap I)\ge1-\eps
\]
for every interval $I\subseteq\R$ of length $1$, and
\[
x+sK\not\subseteq E
\qquad
\text{for every }x\in\R\text{ and every }s\in\R\setminus\{0\}.
\]
\end{proposition}

\begin{proof}
Fix $\eps\in(0,1)$. For $k\ge1$, set
\[
\eta_k=\eps 2^{-k-1}.
\]
Choose an integer $q_k$ satisfying
\[
q_k\ge2^k\Lambda(\eta_k/2),
\]
and put
\[
\mathcal B_k=\mathcal B(q_k,\eta_k).
\]
Define
\[
E=\R\setminus\bigcup_{k=1}^{\infty}\mathcal B_k.
\]
Every $\mathcal B_k$ is open and $1$-periodic. Therefore $E$ is closed and
$1$-periodic.

Let $I$ be an interval of length $1$. By \eqref{eq:blocker-measure}, periodicity, and
countable subadditivity,
\[
\begin{aligned}
m(I\setminus E)
&=m\left(I\cap\bigcup_{k=1}^{\infty}\mathcal B_k\right)\\
&\le\sum_{k=1}^{\infty}m(\mathcal B_k\cap I)\\
&=\sum_{k=1}^{\infty}2\eta_k
=\eps.
\end{aligned}
\]
Thus $m(E\cap I)\ge1-\eps$.

Fix $x\in\R$ and $s\in\R\setminus\{0\}$. Choose $k$ so large that
\[
2^{-k}\le\abs{s}.
\]
Then
\[
\abs{s}q_k
\ge2^{-k}q_k
\ge\Lambda(\eta_k/2).
\]
The hypothesis on $K$ shows that $\abs{s}q_kK\bmod1$ is
$\eta_k/2$-dense. Lemma~\ref{lem:block} gives
\[
(x+sK)\cap\mathcal B_k\ne\emptyset.
\]
Since $E\cap\mathcal B_k=\emptyset$, it follows that $x+sK\not\subseteq E$.

It remains only to verify that $E$ is nowhere dense. Since $q_k\ge2^k$, the sequence
$q_k$ tends to infinity. Let $J\subseteq\R$ be a nonempty open interval. For all
sufficiently large $k$, the lattice $q_k^{-1}\Z$ meets $J$. Every point of this lattice
belongs to $\mathcal B_k$, so $J$ meets $\bigcup_k\mathcal B_k$. The union of the
blockers is therefore open and dense. Its complement $E$ is closed and nowhere dense.
\end{proof}

\begin{remark}\label{rem:density-beyond-rajchman}
The hypothesis of Proposition~\ref{prop:multiscale} is geometric, and it is strictly
weaker than the condition that $K$ supports a Rajchman measure. To see this, let
$(M_n)_{n\ge1}$ be any sequence of positive integers tending to infinity, and set
\[
K
=
\{0\}
\cup
\bigcup_{n=1}^{\infty}
\left\{
\frac{j}{2^nM_n}:0\le j\le M_n
\right\}.
\]
The set $K$ is compact: every point in its $n$th finite block lies in $[0,2^{-n}]$, so the
only possible accumulation point not already in a fixed block is $0$. It is also countable,
and hence it cannot support a Rajchman probability measure by
Proposition~\ref{prop:diffuse}.

Nevertheless, $K$ satisfies the hypothesis of Proposition~\ref{prop:multiscale}. Given
$\tau\ge2$, choose $n$ so that
\[
2^n\le\tau<2^{n+1},
\]
and put $t=\tau/2^n\in[1,2)$. For every $y\in[0,1]$, choose an integer
$j\in\{0,\ldots,M_n\}$ nearest to $M_ny/t$. Then
\[
\left|y-\frac{tj}{M_n}\right|
\le\frac{t}{2M_n}
\le\frac1{M_n}.
\]
Since
\[
\frac{tj}{M_n}
=
\tau\frac{j}{2^nM_n}
\in\tau K,
\]
and
\[
d_{\T}\left(y\bmod1,\frac{tj}{M_n}\bmod1\right)
\le \left|y-\frac{tj}{M_n}\right|.
\]
It follows that $\tau K\bmod1$ is $M_n^{-1}$-dense in $\T$. As $\tau$ tends to
infinity, so does $n$, and therefore $M_n^{-1}$ tends to zero. More explicitly, given
$\delta>0$, choose $N$ so that $M_n^{-1}\le\delta$ for every $n\ge N$; then
$\Lambda(\delta)=2^N$ works in Proposition~\ref{prop:multiscale}.

This example also shows that irrational differences are not necessary for the blocking
argument, since every element of $K-K$ is rational. Conversely, even a very large
difference set does not imply the required density property. The middle-thirds Cantor set
$C$ satisfies
\[
C-C=[-1,1],
\]
as follows, for example, from balanced ternary expansions. On the other hand,
\[
3^nC\bmod1\subseteq C\bmod1
\]
for every $n$, and these dilates omit a fixed neighborhood of $1/2$; see
Example~\ref{ex:cantor-no-rajchman}. Thus the static size of $K-K$ does not capture the
geometric property used in Proposition~\ref{prop:multiscale}. Rajchman decay is one
natural mechanism which guarantees that property, but it is not the only one.
\end{remark}

\begin{proof}[Proof of Theorem~\ref{thm:main}]
Choose a Rajchman probability measure $\mu$ such that
\[
\supp\mu\subseteq A.
\]
Since $\mu(\supp\mu)=1$, inner regularity gives a compact set
\[
F\subseteq\supp\mu
\qquad\text{with}\qquad
\mu(F)>0.
\]
By Lemma~\ref{lem:band}, the normalized restriction
\[
\nu=\frac{\mu|_F}{\mu(F)}
\]
is a Rajchman probability measure. Put
\[
K=\supp\nu.
\]
Then $K$ is a nonempty compact subset of $F$, and hence $K\subseteq A$.
Lemma~\ref{lem:equi} shows that $K$ satisfies the hypothesis of
Proposition~\ref{prop:multiscale}. We obtain a closed, $1$-periodic, nowhere dense set
$E\subseteq\R$ which is $(1-\eps)$-large and contains no affine copy of $K$.

If $E$ contained an affine copy $x+sA$, then it would contain the corresponding affine
copy $x+sK$, since $K\subseteq A$. This contradiction completes the proof.
\end{proof}

\begin{proof}[Proof of Corollary~\ref{cor:compact-avoiding}]
Let $E_0\subseteq\R$ be the set supplied by Theorem~\ref{thm:main}, and put
\[
E=E_0\cap[0,1].
\]
The set $E$ is compact, and the local measure estimate gives $m(E)\ge1-\eps$.
Since $E\subseteq E_0$, it contains no affine copy of $A$.
\end{proof}

\section{Examples and basic structure of Rajchman supports}\label{sec:examples}

\subsection{Necessary topological structure}

\begin{proposition}\label{prop:diffuse}
Every Rajchman probability measure on $\R$ is atomless. Consequently its support is
uncountable and has no isolated points. Thus its support is a nonempty perfect closed set.
\end{proposition}

\begin{proof}
We include the continuous form of Wiener's argument \cite{Wiener}. For $T>0$, Fubini's
theorem gives
\[
\frac{1}{2T}\int_{-T}^{T}\abs{\fhat\mu(\xi)}^2\,d\xi
=\iint_{\R^2}D_T(x-y)\,d\mu(x)\,d\mu(y),
\]
where
\[
D_T(u)
=\frac{1}{2T}\int_{-T}^{T}e^{-2\pi i\xi u}\,d\xi
=
\begin{cases}
\dfrac{\sin(2\pi Tu)}{2\pi Tu},&u\ne0,\\[5pt]
1,&u=0.
\end{cases}
\]
For every $u\in\R$, $\abs{D_T(u)}\le1$, and $D_T(u)$ tends to
$\mathbf 1_{\{0\}}(u)$ as $T$ tends to infinity. Dominated convergence therefore yields
\[
\lim_{T\to\infty}\frac{1}{2T}\int_{-T}^{T}\abs{\fhat\mu(\xi)}^2\,d\xi
=(\mu\times\mu)\bigl(\{(x,y):x=y\}\bigr)
=\sum_{a\in\R}\mu(\{a\})^2.
\]
The last equality follows from Fubini applied to the sections of the diagonal; the set of
atoms of a finite measure is at most countable.

Since $\abs{\fhat\mu(\xi)}\le1$, the Rajchman property forces the expression on the left
to tend to zero. Indeed, given $\delta>0$, choose $R>0$ such that
$\abs{\fhat\mu(\xi)}^2<\delta$ for $\abs{\xi}>R$. For $T>R$,
\[
\frac{1}{2T}\int_{-T}^{T}\abs{\fhat\mu(\xi)}^2\,d\xi
\le\frac{R}{T}+\delta.
\]
It follows that $\mu(\{a\})=0$ for every $a\in\R$.

Put $K=\supp\mu$. Since $\mu(\R\setminus K)=0$, a countable $K$ would give
\[
1=\mu(K)=\sum_{x\in K}\mu(\{x\})=0,
\]
a contradiction. Hence $K$ is uncountable. If $x\in K$ were isolated in $K$, some open
interval $U$ containing $x$ would satisfy $U\cap K=\{x\}$. Then
$\mu(U)=\mu(\{x\})=0$, contrary to $x\in\supp\mu$. Thus $K$ has no isolated points.
It is nonempty and closed by the general properties of supports, completing the proof.
\end{proof}

In particular, every set covered by Theorem~\ref{thm:main} contains an uncountable perfect
subset. The converse is false, even for the most familiar perfect set. The following fact is
classical: Rajchman proved that every $H$-set is annihilated by every Rajchman measure,
and the middle-thirds Cantor set is an $H$-set \cite[Section~2]{Lyons}. We include a direct
verification because it displays the same obstruction to equidistribution that underlies our
main argument.

\begin{example}\label{ex:cantor-no-rajchman}
The middle-thirds Cantor set $C$ is a perfect uncountable compact set, but no Rajchman
probability measure has support contained in $C$.

\begin{proof}[Verification]
The compactness, uncountability, and perfectness of $C$ are standard. Let
$\pi:\R\to\T$ be the quotient map. Multiplication by $3^n$ deletes the first $n$
ternary digits modulo one: if
\[
x=\sum_{j\ge1}2\eps_j3^{-j}\in C,
\qquad
\eps_j\in\{0,1\}\text{ for every }j\ge1,
\]
then
\[
\pi(3^nx)
=\pi\left(\sum_{j>n}\frac{2\eps_j}{3^{j-n}}\right)
\in\pi(C).
\]
Consequently
\[
3^nC\bmod1\subseteq\pi(C)
\qquad\text{for every }n\ge1.
\]
The point $\pi(1/2)$ has distance at least $1/6$ from $\pi(C)$, because
\[
C\subseteq[0,1/3]\cup[2/3,1].
\]

Suppose now that a Rajchman probability measure $\mu$ had
$K=\supp\mu\subseteq C$. The compact set $K$ satisfies Lemma~\ref{lem:equi}. Applying
that lemma with $\eta=1/7$, we find that $3^nK\bmod1$ is $1/7$-dense in $\T$ for all
sufficiently large $n$. But
\[
3^nK\bmod1\subseteq3^nC\bmod1\subseteq\pi(C),
\]
and no subset of $\pi(C)$ is $1/7$-dense because it stays at distance at least $1/6$ from
$\pi(1/2)$. This contradiction proves the claim.
\end{proof}
\end{example}

\subsection{Standard sources of Rajchman measures}

We record several standard sources of Rajchman measures.

\begin{proposition}\label{prop:families}
Each of the following classes consists of sets supporting Rajchman probability measures.
\begin{enumerate}
\item Lebesgue measurable sets of positive Lebesgue measure;
\item compact sets of positive Fourier dimension, and hence compact Salem sets of positive
dimension;
\item the central Cantor sets $C_r$, $0<r<\tfrac12$, for which $r^{-1}$ is not a Pisot
number (a Pisot number is a real algebraic integer greater than $1$ all of whose other
algebraic conjugates have modulus less than $1$);
\item non-singleton self-similar sets generated by finitely many maps of the form
\[
f_j(x)=r_jx+b_j,
\qquad\text{where }0<r_j<1,
\]
for which
\[
\frac{\log r_j}{\log r_\ell}\notin\mathbb Q
\]
for at least one pair $j\ne\ell$.
\end{enumerate}
There also exist Cantor sets of Hausdorff dimension zero supporting Rajchman measures.
\end{proposition}

\begin{proof}
Let $A$ be Lebesgue measurable with $m(A)>0$. By restricting first to a bounded interval
and then using inner regularity, choose a compact set $F\subseteq A$ with
$0<m(F)<\infty$. The probability measure
\[
d\mu(x)=\frac{\mathbf 1_F(x)}{m(F)}\,dx
\]
has $\supp\mu\subseteq F\subseteq A$. Its density $\mathbf 1_F/m(F)$ belongs to
$L^1(\R)$, so the Riemann--Lebesgue lemma gives
\[
\fhat\mu(\xi)
=\frac{1}{m(F)}\int_F e^{-2\pi i\xi x}\,dx
\xrightarrow[\abs{\xi}\to\infty]{}0.
\]
Thus $\mu$ is Rajchman, proving (1).

If the compact set $A$ has positive Fourier dimension, then by definition there exist a
probability measure $\mu$ with $\supp\mu\subseteq A$ and constants $C>0$ and
$\alpha>0$ such that
\[
\abs{\fhat\mu(\xi)}
\le C(1+\abs{\xi})^{-\alpha}
\qquad\text{for all }\xi\in\R.
\]
Such a measure is Rajchman, proving (2).

For $0<r<\tfrac12$, let $C_r$ be the attractor of
\[
f_0(x)=rx,
\qquad
f_1(x)=rx+(1-r).
\]
Let $\mu_r$ be the equal-weight self-similar probability measure on $C_r$. Then
$\supp\mu_r=C_r$, and Salem's theorem shows that $\mu_r$ is Rajchman when $r^{-1}$ is
not a Pisot number; see
\cite[p.~344]{LiSahlsten} and \cite{Salem}. This proves (3).

For the systems in (4), choose any probability vector with all weights positive and let
$\mu$ be the associated self-similar measure. Its support is the entire attractor.
After an affine normalization of the attractor, Li and Sahlsten's theorem applies; they
proved, without a separation assumption, that every such self-similar measure is Rajchman
whenever two logarithmic contraction ratios have irrational quotient
\cite[Theorem~1.2]{LiSahlsten}. 

Finally, Bluhm constructed a Cantor set of Hausdorff dimension zero carrying a Rajchman
probability measure \cite{Bluhm}. This example shows that neither positive Lebesgue measure
nor positive Hausdorff or Fourier dimension is necessary.
\end{proof}

\section{Logarithmic size and thin Rajchman supports}\label{sec:log-size}

\subsection{Logarithmic dimensions}
\label{subsec:log-dimensions}

A \emph{Hausdorff gauge} is a nondecreasing function
$h:[0,\infty)\to[0,\infty)$ such that $h(0)=0$, $h(r)>0$ for $r>0$, and
$h(r)\xrightarrow[r\searrow0]{}0$. For $\delta>0$ and $A\subseteq\R$, define
\[
\mathcal H_\delta^h(A)
=\inf\left\{
\sum_{j\in J}h(\operatorname{diam}U_j):
A\subseteq\bigcup_{j\in J}U_j,
\quad
\operatorname{diam}U_j\le\delta
\right\},
\]
where the infimum is over finite or countable covers by nonempty sets
$U_j\subseteq\R$. The generalized Hausdorff measure associated with $h$ is
\[
\mathcal H^h(A)
=\lim_{\delta\searrow0}\mathcal H_\delta^h(A)
=\sup_{\delta>0}\mathcal H_\delta^h(A).
\]

For packing measure, we likewise evaluate $h$ at diameters. Write
$B(x,r)=[x-r,x+r]$. A
$\delta$-packing of a set $F\subseteq\R$ is a finite or countable pairwise disjoint
family $\{B(x_j,r_j)\}_j$, where $x_j\in F$ and $0<2r_j\le\delta$, and its $h$-cost is
$\sum_jh(2r_j)$. The supremum of these costs is $\mathcal P_\delta^h(F)$, and
$\mathcal P_0^h(F)=\lim_{\delta\searrow0}\mathcal P_\delta^h(F)$. The associated
packing measure is
\[
\mathcal P^h(F)
=\inf\left\{
\sum_{j=1}^{\infty}\mathcal P_0^h(F_j):
F\subseteq\bigcup_{j=1}^{\infty}F_j
\right\}.
\]
For $s>0$, let $h_s$ be any Hausdorff gauge satisfying
\[
h_s(r)=\bigl(\log(1/r)\bigr)^{-s}
\]
for all sufficiently small $r>0$. Define
\[
\begin{aligned}
\dimHlog A&=\inf\{s>0:\mathcal H^{h_s}(A)=0\},\\
\dimPlog A&=\inf\{s>0:\mathcal P^{h_s}(A)=0\}.
\end{aligned}
\]
Here $\inf\varnothing=\infty$. These are the logarithmic Hausdorff and packing dimensions
used in \cite{ShmerkinYavicoli}. Replacing $h(\operatorname{diam}U_j)$ by
$h(\operatorname{diam}U_j/2)$ in the Hausdorff construction, and $h(2r_j)$ by
$h(r_j)$ in the packing construction, gives the same logarithmic dimensions.

For a nonempty compact set $K\subseteq\R$, let $N(K,r)$ be the least number of closed
intervals of radius $r$ needed to cover $K$. Its upper logarithmic box dimension is
\[
\overdimBlog K
=\limsup_{r\searrow0}
\frac{\log N(K,r)}{\log\log(1/r)}.
\]

For comparison, we recall the one-dimensional affine case of
\cite[Lemma~1.1 and Theorem~1.3]{ShmerkinYavicoli}.

\begin{theorem}[Shmerkin--Yavicoli]
\label{thm:shmerkin-yavicoli}
Let $A\subseteq\R$ be Borel. If
\[
\dimHlog A>1
\qquad\text{or}\qquad
\dimPlog A>2,
\]
then there is a conull Lebesgue measurable set $E\subseteq\R$ containing no affine copy
of $A$.
\end{theorem}

\subsection{Lower bounds forced by Fourier decay}

The next elementary obstruction is useful when comparing sets supporting Rajchman
measures with logarithmic-dimension criteria.

\begin{proposition}\label{prop:logbox-lower}
If a compact set $K$ supports a Rajchman probability measure, then
$\overdimBlog K\ge1$.
\end{proposition}

\begin{proof}
Suppose instead that $\overdimBlog K<1$, and choose
\[
\overdimBlog K<\theta<1.
\]
Put $L=\log(1/r)$, so that $r=e^{-L}$. For all sufficiently small $r$,
\[
M:=N(K,r)\le L^\theta.
\]
Choose centers $x_1,\ldots,x_M$ of a cover of $K$ by intervals of radius $r$, and let
$Q=\lfloor L\rfloor$. For small $r$ we have $Q\ge2$. We use the following simultaneous
form of Dirichlet's approximation lemma
\cite[Chapter~II, (1.1), p.~27]{Schmidt}: given $z_1,\ldots,z_M\in\R$ and an integer
$Q\ge2$, there is an integer $q$ with
$1\le q<Q^M$ and $\tnorm{qz_j}\le1/Q$ for every $j$ with $1\le j\le M$.
Applying the lemma to $x_1,\ldots,x_M$, choose $q$ such that
\begin{equation}\label{eq:q-dirichlet}
1\le q<Q^M,
\qquad
\tnorm{qx_j}\le\frac1Q
\quad\text{for every }j\text{ with }1\le j\le M.
\end{equation}
Let $q(r)$ be the maximal integer $q$ satisfying \eqref{eq:q-dirichlet}.

Given $x\in K$, choose $j$ with $\abs{x-x_j}\le r$. Then
\[
\tnorm{qx}
\le\tnorm{qx_j}+q\abs{x-x_j}
\le\frac1Q+Q^Mr.
\]
Since $M\le L^\theta$, $Q\le L$, and $r=e^{-L}$,
\[
Q^Mr
\le\exp\bigl(L^\theta\log L-L\bigr)
\xrightarrow[r\searrow0]{}0.
\]
Let $\mu$ be a Rajchman probability measure with $\supp\mu\subseteq K$. The elementary
inequality $\abs{e^{-2\pi it}-1}\le2\pi\tnorm{t}$ gives
\[
\abs{\fhat\mu(q)-1}
\le2\pi\sup_{x\in K}\tnorm{qx}
\xrightarrow[r\searrow0]{}0.
\]
Here and below $q=q(r)$. The values $q(r)$ are unbounded in every neighborhood of zero.
Indeed, otherwise they would be bounded for all sufficiently small $r$, so some fixed
$q_0\ge1$ would occur along a sequence tending to zero. The preceding estimate would
then force $\sup_{x\in K}\tnorm{q_0x}=0$, and hence
$q_0K\subseteq\Z$. Since $K$ is bounded, it would be finite, contradicting
Proposition~\ref{prop:diffuse}. We may therefore choose a sequence $(r_k)_k$ with
$r_k\searrow0$ such that $\bigl(q(r_k)\bigr)_k$ tends to infinity. Along the same
sequence, $\bigl(\fhat\mu(q(r_k))\bigr)_k$ tends to $1$, contradicting the Rajchman
property.
\end{proof}

\begin{corollary}\label{cor:logpacking-lower}
If a compact set $K$ supports a Rajchman probability measure, then
\[
\dimPlog K\ge1.
\]
\end{corollary}

\begin{proof}
Suppose that $\dimPlog K<1$. By the definition of logarithmic packing
dimension, there is a $t<1$ such that
\[
\mathcal P^{h_t}(K)=0.
\]
Choose $s$ with $t<s<1$. For all sufficiently small $r$, we have
$h_s(r)\le h_t(r)$, and hence monotonicity of packing measure gives
$\mathcal P^{h_s}(K)=0$ as well.
By the definition of packing measure, there is a countable cover
$K\subseteq\bigcup_jA_j$ such that
\[
\sum_j\mathcal P_0^{h_s}(A_j)<1.
\]
Replace $A_j$ by $E_j=A_j\cap K$ and set $F_j=\overline{E_j}$. Then each $F_j$ is
compact, $F_j\subseteq K$, the sets $F_j$ still cover $K$, and
\[
C_j:=\mathcal P_0^{h_s}(E_j)<\infty.
\]
Discard any empty $E_j$.

We claim that
\[
\overdimBlog F_j\le s
\qquad\text{for every }j.
\]
Fix $j$. By the definition of packing premeasure, there is $\delta_j>0$ such
that every $\delta_j$-packing $\{B(x_\ell,r_\ell)\}_\ell$ of $E_j$ satisfies
\[
\sum_\ell h_s(2r_\ell)\le C_j+1.
\]
For $0<2r\le\delta_j$, take a maximal disjoint family
\[
\{B(x_\ell,r)\}_{\ell=1}^{M},
\qquad\text{with }x_\ell\in E_j.
\]
Then
\[
M h_s(2r)\le C_j+1.
\]
By maximality, the closed balls $B(x_\ell,2r)$ cover $E_j$ and therefore also
its closure $F_j$. Consequently
\[
N(F_j,2r)
\le M
\le\frac{C_j+1}{h_s(2r)}
=(C_j+1)\bigl(\log(1/(2r))\bigr)^s.
\]
Taking logarithms, dividing by $\log\log(1/(2r))$, and letting $r$ tend to zero
proves the claim.

Let $\mu$ be a Rajchman probability measure with $\supp\mu\subseteq K$. Since
$K\subseteq\bigcup_jF_j$ and $\mu(K)=1$, there is an index $j$ for which
$\mu(F_j)>0$. By Lemma~\ref{lem:band}, the normalized restriction
\[
\nu=\frac{\mu|_{F_j}}{\mu(F_j)}
\]
is Rajchman.

Now $\supp\nu\subseteq F_j$, so Proposition~\ref{prop:logbox-lower} applied to
$F_j$ gives
\[
\overdimBlog F_j\ge1,
\]
contradicting the claim above because $s<1$.
\end{proof}

\subsection{Thin supports at the critical endpoint}

The theorem of Iva\v{s}ev-Musatov already gives, for every Hausdorff gauge $h$, a compact
$h$-null set supporting a Rajchman measure \cite{IvashevMusatov,Lyons}; K\"orner proved
substantial refinements \cite{KornerIvashev}. The purpose of this subsection is different.
We impose, at the same time, the smallest possible logarithmic upper box and packing
dimensions. The finite-sampling lemma below is the main ingredient in the construction.

\begin{lemma}\label{lem:finite-sampling}
Let $0<\eta<1$, let $R\ge2$, and suppose that
\[
\nu=\frac1J\sum_{v=1}^J\nu_v,
\]
where the $\nu_v$ are probability measures on $[0,1]$. There is an absolute
constant $C>0$ with the following property. If $b\ge1$ is an integer and
\[
M:=bJ\ge C\eta^{-2}\log(R/\eta),
\]
then there are points $x_{k,v}\in\supp\nu_v$, $1\le k\le b$ and $1\le v\le J$,
such that the finitely supported probability measure
\[
\sigma=\frac1M\sum_{v=1}^J\sum_{k=1}^b\delta_{x_{k,v}}
\]
satisfies
\[
\sup_{\abs\xi\le R}\abs{\fhat\sigma(\xi)-\fhat\nu(\xi)}\le\eta.
\]
Moreover, suppose that $b=2$ and that each $\nu_v$ has a density $\phi_v$ with
$\lVert\phi_v\rVert_\infty\le L$. If
\[
0<r\le\frac1{40LJ},
\]
then the points may be chosen so that
\[
\abs{x_{1,v}-x_{2,v}}>5r
\qquad\text{for every }v\text{ with }1\le v\le J.
\]
\end{lemma}

\begin{proof}
For each $v$, let $X_{1,v},\ldots,X_{b,v}$ be independent random variables with law
$\nu_v$, and assume that all these variables are mutually independent. Define the random
probability measure
\[
\sigma=\frac1M\sum_{v=1}^J\sum_{k=1}^b\delta_{X_{k,v}}.
\]
For every fixed $\xi\in\R$,
\[
\mathbb E[\fhat\sigma(\xi)]
=\frac1M\sum_{v=1}^J\sum_{k=1}^b
\mathbb E\bigl[e^{-2\pi i\xi X_{k,v}}\bigr]
=\fhat\nu(\xi).
\]

We use Hoeffding's inequality
\cite[Theorem~2, inequality~(2.6), p.~16]{Hoeffding}, applied also to the negatives of
the random variables. If $Y_1,\ldots,Y_M$ are independent real random variables with
$a_\ell\le Y_\ell\le b_\ell$ almost surely for every $\ell$ with $1\le\ell\le M$, then,
for $t>0$,
\[
\mathbb P\left(
\abs{\frac1M\sum_{\ell=1}^M(Y_\ell-\mathbb EY_\ell)}>t
\right)
\le2\exp\left(
-\frac{2M^2t^2}{\sum_{\ell=1}^M(b_\ell-a_\ell)^2}
\right).
\]
Apply this separately to the real and imaginary parts of
$e^{-2\pi i\xi X_{k,v}}$, both of which lie in $[-1,1]$. If a complex number has
modulus greater than $\eta/2$, then one of its real and imaginary parts has absolute
value greater than $\eta/(2\sqrt2)$. A union bound therefore gives
\begin{equation}\label{eq:hoeffding-fourier}
\mathbb P\left(
\abs{\fhat\sigma(\xi)-\fhat\nu(\xi)}>\frac\eta2
\right)
\le4\exp\left(-\frac{M\eta^2}{16}\right).
\end{equation}

For any probability measure $\rho$ on $[0,1]$,
\[
\begin{aligned}
\abs{\fhat\rho(\xi)-\fhat\rho(\zeta)}
&\le\int_{[0,1]}
\abs{e^{-2\pi i\xi x}-e^{-2\pi i\zeta x}}\,d\rho(x)\\
&\le2\pi\abs{\xi-\zeta}.
\end{aligned}
\]
Indeed, $\abs{e^{iu}-e^{iv}}\le\abs{u-v}$ and $0\le x\le1$. Thus
$\fhat\sigma-\fhat\nu$ is $4\pi$-Lipschitz.

Set $\kappa=\eta/(16\pi)$ and choose a finite grid $\Gamma\subseteq[-R,R]$ containing the
endpoints and having mesh at most $\kappa$, where the mesh is the largest distance between
consecutive grid points. We may arrange that
\[
\#\Gamma\le2+\frac{2R}{\kappa}=2+\frac{32\pi R}{\eta},
\]
and every $\xi\in[-R,R]$ lies within distance $\kappa$ of some $\gamma\in\Gamma$. Put
$F(\xi)=\fhat\sigma(\xi)-\fhat\nu(\xi)$. If
$\sup_{\abs\xi\le R}\abs{F(\xi)}>\eta$, then for some $\xi\in[-R,R]$ and a grid point
$\gamma$ with $\abs{\xi-\gamma}\le\kappa$,
\[
\abs{F(\gamma)}
\ge\abs{F(\xi)}-4\pi\abs{\xi-\gamma}
>\eta-4\pi\kappa
=\frac{3\eta}{4}
>\frac\eta2.
\]
Consequently, the union bound and \eqref{eq:hoeffding-fourier} give
\[
\begin{aligned}
\mathbb P\left(
\sup_{\abs\xi\le R}\abs{\fhat\sigma(\xi)-\fhat\nu(\xi)}>\eta
\right)
&\le\sum_{\gamma\in\Gamma}
\mathbb P\left(\abs{F(\gamma)}>\frac\eta2\right)\\
&\le C_0\frac R\eta
\exp\left(-\frac{M\eta^2}{16}\right)
\end{aligned}
\]
for an absolute constant $C_0$. The assumption on $M$ makes the last expression at
most $C_0(R/\eta)^{1-C/16}$, which is less than $1/4$ if the absolute constant $C$ in
the statement is sufficiently large. Moreover, $X_{k,v}\in\supp\nu_v$ almost surely.
Hence some realization gives the first assertion.

Now assume the hypotheses of the final assertion. For every $v$,
\[
\begin{aligned}
\mathbb P\bigl(\abs{X_{1,v}-X_{2,v}}\le5r\bigr)
&=\int_{[0,1]}\nu_v\bigl([x-5r,x+5r]\bigr)\,d\nu_v(x)\\
&\le10Lr.
\end{aligned}
\]
The inequality uses the density bound and the fact that the interval has length $10r$.
The union bound over $v=1,\ldots,J$ shows that the probability of a separation failure
for at least one $v$ is at most
\[
10LJr\le\frac14.
\]
Thus the probability that either the Fourier approximation or one of the separation
conditions fails is less than $1/2$. A realization for which neither failure occurs
supplies all the required points simultaneously.
\end{proof}

\begin{theorem}\label{thm:thin-support}
Let $h$ be a Hausdorff gauge. There exist
a compact set $K\subseteq(0,1)$ and a Rajchman probability measure $\mu$ such that
\[
\supp\mu=K,\qquad
\mathcal H^h(K)=0,\qquad
\overdimBlog K=\dimPlog K=1.
\]
In particular, there exist a compact set $K\subseteq(0,1)$ and a Rajchman probability
measure $\mu$ such that
\[
\supp\mu=K,\qquad
\dimHlog K=0,\qquad
\dimPlog K=1.
\]
\end{theorem}

\begin{proof}
We divide the construction into three steps. The first produces a Rajchman measure whose
support is a binary Moran set at the critical logarithmic scale. The second obtains the
sharp box and packing estimates from the cylinder counts and masses. The last step
extracts an $h$-null support without increasing its packing dimension.

\smallskip
\noindent\emph{Step 1: a binary Moran Rajchman measure.}
Fix a nonnegative function
\[
\phi\in C_c^\infty((-1,1)),
\qquad
\int_{\R}\phi(x)\,dx=1,
\]
and write $\phi_r(x)=r^{-1}\phi(x/r)$. Put
\[
C_\phi=2\pi\int_{\R}\abs x\phi(x)\,dx.
\]
Then
\begin{equation}\label{eq:phi-near-zero}
\abs{1-\fhat\phi(t)}\le C_\phi\abs t.
\end{equation}
This follows from $\abs{1-e^{-2\pi itx}}\le2\pi\abs{tx}$ and the definition of
$C_\phi$.

Since $\fhat\phi$ is a Schwartz function, there is a constant $C_2$ such that
\[
\abs{\fhat\phi(t)}\le C_2(1+\abs t)^{-2}.
\]
Let $n_0$ be an integer to be chosen below, and set
\[
\eps_n=(n+n_0)^{-2}
\qquad\text{for all }n\ge0.
\]

Fix a sufficiently large constant $a$ and set $A_n=a(n+n_0)$. Then $A_n\ge2$ and
\begin{equation}\label{eq:phi-tail}
\sup_{\abs t\ge A_n}\abs{\fhat\phi(t)}\le\eps_n
\qquad\text{for all }n\ge1
\end{equation}
and $\log A_n=O(\log(n+n_0))$.

Let $C$ be the constant in Lemma~\ref{lem:finite-sampling}. Choose a sufficiently small absolute constant $\alpha>0$, then take $n_0$ sufficiently large and the integer $B$
sufficiently large. Define
\[
M_n=B2^n,\qquad
L_n=\alpha\eps_n^2M_n
=\frac{\alpha B2^n}{(n+n_0)^4},
\qquad
D_n=e^{L_n},
\qquad
r_n=\frac{A_{n+1}}{D_n}.
\]
Here $M_n$ is the number of level-$n$ cylinders, $\eps_n$ is the Fourier-error budget,
$D_n$ is the frequency cutoff, and $r_n$ is the smoothing scale. Notice that
$r_nD_n=A_{n+1}$; hence, for $\abs\xi\ge D_n$, \eqref{eq:phi-tail} controls the factor
$\fhat\phi(r_n\xi)$.
The choices can be made so that $L_n$ is increasing, $D_0\ge2$, and, for every
$n\ge1$,
\begin{equation}\label{eq:parameter-conditions}
\begin{aligned}
C_\phi r_nD_{n-1}&\le\eps_n,
&5r_n&<r_{n-1},\\
M_{n-1}\frac{r_n}{r_{n-1}}
&\le\min\left\{2^{-n-10},\frac1{40\lVert\phi\rVert_\infty}\right\},
&M_n&\ge C\eps_n^{-2}\log(D_n/\eps_n).
\end{aligned}
\end{equation}
For completeness, let $\Delta_n=L_n-L_{n-1}$. After taking $n_0$ large,
\[
\Delta_n\asymp\frac{\alpha B2^n}{(n+n_0)^4}>0
\]
and the definitions give
\[
 r_nD_{n-1}=A_{n+1}e^{-\Delta_n},
 \qquad
\frac{r_n}{r_{n-1}}=\frac{A_{n+1}}{A_n}e^{-\Delta_n}.
\]
Thus a sufficiently large increment $\Delta_n$ makes both $r_nD_{n-1}$ and
$r_n/r_{n-1}$ small. More precisely, $\Delta_n$ eventually dominates
\[
\log M_{n-1}+\log A_n+\log A_{n+1}
+\log(1/\eps_n)+n+\log\bigl(1+\lVert\phi\rVert_\infty\bigr).
\]
Indeed, for every fixed $B>0$,
\[
\Delta_n-\log M_{n-1}\xrightarrow[n\to\infty]{}+\infty,
\]
and the other terms displayed above grow at most linearly or logarithmically in $n$.
After fixing $\alpha$ and $n_0$, increasing $B$ handles the finitely many remaining
indices and gives all the conditions in
\eqref{eq:parameter-conditions} except the final sampling-size condition. For that
condition, observe that
\[
C\eps_n^{-2}\log(D_n/\eps_n)
=C\alpha M_n+C\eps_n^{-2}\log(1/\eps_n).
\]
Taking $\alpha<(2C)^{-1}$ and then increasing $B$ makes this at most $M_n$ for every
$n$. Finally, as $B$ tends to infinity, $D_0$ tends to infinity while $Br_0$ tends to
zero. Indeed,
\[
D_0=e^{\alpha B/n_0^4},
\qquad
Br_0=BA_1e^{-\alpha B/n_0^4}.
\]
We may therefore also arrange $D_0\ge2$ and $4Br_0<1$.

Choose points $c_{0,1},\ldots,c_{0,B}$ such that the closed intervals
\[
I_{0,v}=[c_{0,v}-2r_0,c_{0,v}+2r_0],
\qquad 1\le v\le B,
\]
are pairwise disjoint and contained in $(0,1)$; this is possible because
$4Br_0<1$. We call these intervals the level-$0$ cylinders and set
\[
\mu_0=\frac1B\sum_{v=1}^{B}\phi_{r_0}(x-c_{0,v})\,dx.
\]

Suppose inductively that the level-$(n-1)$ cylinders have centers
$c_{n-1,1},\ldots,c_{n-1,M_{n-1}}$ and that
\[
\mu_{n-1}=\frac1{M_{n-1}}\sum_{v=1}^{M_{n-1}}\nu_{n-1,v},
\qquad
\nu_{n-1,v}=\phi_{r_{n-1}}(x-c_{n-1,v})\,dx.
\]
Each $\nu_{n-1,v}$ is a probability measure supported in the interval of radius
$r_{n-1}$ about $c_{n-1,v}$, and hence inside the corresponding level-$(n-1)$
cylinder. We apply Lemma~\ref{lem:finite-sampling} with
\[
J=M_{n-1},\qquad b=2,\qquad \eta=\eps_n,\qquad R=D_n.
\]
The density of each $\nu_{n-1,v}$ is bounded by
$\lVert\phi\rVert_\infty/r_{n-1}$, and the inequality
\[
M_{n-1}\frac{r_n}{r_{n-1}}
\le\frac1{40\lVert\phi\rVert_\infty}
\]
is equivalent to the separation hypothesis of the lemma with $r=r_n$. The final
inequality in \eqref{eq:parameter-conditions} is its sampling-size hypothesis. Thus the
lemma provides points
\[
x_{k,v}^{(n)}\in\supp\nu_{n-1,v},
\qquad
1\le k\le2,
\quad
1\le v\le M_{n-1}.
\]
Define
\[
\sigma_n
=\frac1{M_n}
\sum_{v=1}^{M_{n-1}}\sum_{k=1}^{2}\delta_{x_{k,v}^{(n)}}.
\]
Since $M_n=2M_{n-1}$, this is a probability measure. The averaged input measure in
Lemma~\ref{lem:finite-sampling} is
\[
\frac1{M_{n-1}}\sum_{v=1}^{M_{n-1}}\nu_{n-1,v}=\mu_{n-1}.
\]
Therefore its Fourier-approximation conclusion gives
\begin{equation}\label{eq:sampling-approx}
\sup_{\abs\xi\le D_n}
\abs{\fhat{\sigma_n}(\xi)-\fhat{\mu_{n-1}}(\xi)}
\le\eps_n.
\end{equation}
Its separation conclusion also gives
\[
\abs{x_{1,v}^{(n)}-x_{2,v}^{(n)}}>5r_n
\qquad
\text{for all }1\le v\le M_{n-1}.
\]
Enumerate the sampled points as $c_{n,1},\ldots,c_{n,M_n}$ and define
\[
I_{n,w}=[c_{n,w}-2r_n,c_{n,w}+2r_n],
\qquad 1\le w\le M_n.
\]
These intervals are the level-$n$ cylinders. Set
\[
\mu_n=\sigma_n*\bigl(\phi_{r_n}(x)\,dx\bigr)
=\frac1{M_n}\sum_{w=1}^{M_n}\phi_{r_n}(x-c_{n,w})\,dx.
\]
Thus every atom of $\sigma_n$ is replaced by a smooth probability bump supported in
the interval of radius $r_n$ about that atom, which lies inside the corresponding
level-$n$ cylinder. The two children of a fixed level-$(n-1)$ cylinder are disjoint
because their centers are more than $5r_n$ apart. Moreover,
$x_{k,v}^{(n)}\in\supp\nu_{n-1,v}$ implies
\[
\abs{x_{k,v}^{(n)}-c_{n-1,v}}\le r_{n-1}.
\]
Since $2r_n<r_{n-1}$, every point of this child cylinder is at distance less than
$r_{n-1}+2r_n<2r_{n-1}$ from $c_{n-1,v}$; hence the child cylinder is contained in its
parent. Children belonging to different parents are disjoint because the parent cylinders
are disjoint. Hence the cylinders form a nested binary family. Let $K_n$ be the union of
the level-$n$ cylinders and put
\[
K_*=\bigcap_{n=0}^{\infty}K_n.
\]

We next prove Fourier decay. The three estimates below cover low, intermediate, and high
frequencies: successive measures are close at low frequencies, the previous decay
controls the intermediate annulus, and the smooth factor suppresses high frequencies.
Since $r_{n-1}D_{n-1}=A_n$,
\eqref{eq:phi-tail} gives
\begin{equation}\label{eq:old-high}
\abs{\fhat{\mu_{n-1}}(\xi)}\le\eps_n
\qquad\text{for every }\xi\text{ with }\abs\xi\ge D_{n-1}.
\end{equation}
Since convolution turns into multiplication under the Fourier transform,
\[
\fhat{\mu_n}(\xi)=\fhat{\sigma_n}(\xi)\fhat\phi(r_n\xi).
\]
For $\abs\xi\le D_{n-1}$, add and subtract
$\fhat{\mu_{n-1}}(\xi)\fhat\phi(r_n\xi)$. Since
$\abs{\fhat\phi}\le1$ and $\abs{\fhat{\mu_{n-1}}}\le1$, equations
\eqref{eq:sampling-approx}, \eqref{eq:phi-near-zero}, and
\eqref{eq:parameter-conditions} give
\begin{equation}\label{eq:low-transition}
\begin{aligned}
\abs{\fhat{\mu_n}(\xi)-\fhat{\mu_{n-1}}(\xi)}
&\le \abs{\fhat\phi(r_n\xi)}
 \abs{\fhat{\sigma_n}(\xi)-\fhat{\mu_{n-1}}(\xi)}\\
&\quad+\abs{\fhat{\mu_{n-1}}(\xi)}
 \abs{\fhat\phi(r_n\xi)-1}\\
&\le\eps_n+C_\phi r_nD_{n-1}
\le2\eps_n.
\end{aligned}
\end{equation}
For $D_{n-1}\le\abs\xi\le D_n$, we instead use
$\abs{\fhat\phi}\le1$ to obtain
\begin{equation}\label{eq:new-annulus}
\abs{\fhat{\mu_n}(\xi)}
\le\abs{\fhat{\sigma_n}(\xi)}
\le\abs{\fhat{\mu_{n-1}}(\xi)}+\eps_n
\le2\eps_n.
\end{equation}
Finally, if $\abs\xi\ge D_n$, then $r_n\abs\xi\ge A_{n+1}$, so
\begin{equation}\label{eq:new-high}
\abs{\fhat{\mu_n}(\xi)}
\le\abs{\fhat\phi(r_n\xi)}
\le\eps_{n+1}.
\end{equation}
Thus $\fhat{\mu_n}$ is stable on the old frequency range, is at most $2\eps_n$ on the
new annulus, and has size at most $\eps_{n+1}$ beyond $D_n$.

All the measures are supported in the fixed compact interval $[0,1]$. On any fixed
compact frequency interval, \eqref{eq:low-transition} applies at every sufficiently
large stage because $D_n$ tends to infinity. Since $\sum_n\eps_n<\infty$, the Fourier
transforms are therefore uniformly Cauchy on that interval. Every weakly
convergent subsequence therefore has the same limiting Fourier transform. Compactness of
the space of probabilities on $[0,1]$ and uniqueness of Fourier transforms show that all
weak cluster points coincide. Hence the full sequence converges weakly to a probability
measure $\rho$. For each fixed $n$, every later measure is supported in $K_n$. Since
$K_n$ is closed, the Portmanteau theorem gives
\[
1=\limsup_{m\to\infty}\mu_m(K_n)\le\rho(K_n),
\]
so $\rho(K_n)=1$. Since the sets $K_n$ decrease to $K_*$, continuity from above yields
$\rho(K_*)=1$. If
$D_{n-1}\le\abs\xi\le D_n$, then $\abs\xi\le D_{j-1}$ for every $j>n$, so
\eqref{eq:low-transition} can be telescoped from $\mu_n$ to $\rho$. Together with
\eqref{eq:new-annulus}, this gives
\[
\abs{\fhat\rho(\xi)}
\le2\eps_n+2\sum_{j>n}\eps_j
\xrightarrow[n\to\infty]{}0.
\]
This bound is uniform over the $n$th annulus, whose index tends to infinity with
$\abs\xi$. Thus $\rho$ is Rajchman.

We next compute the cylinder masses, by which we mean the numbers $\rho(I)$ for the
closed level-$n$ intervals $I$ whose union is $K_n$. Every level-$n$ cylinder $I$ has
\[
\rho(I)=M_n^{-1}.
\]
Indeed, if $m\ge n$, then $I$ contains exactly $2^{m-n}$ descendant level-$m$
cylinders. The measure $\mu_m$ places one probability bump of total mass $M_m^{-1}$
inside each such descendant, and hence
\[
\mu_m(I)=2^{m-n}M_m^{-1}=M_n^{-1}.
\]
Choose a continuous function $f$ that equals $1$ on $I$ and $0$ on every other
level-$n$ cylinder. Because $\mu_m$ is supported on $K_n$ for $m\ge n$, its integral
against $f$ is $M_n^{-1}$. Likewise, $\rho(K_n)=1$. Hence weak convergence yields
\[
\rho(I)=\int_{[0,1]}f\,d\rho
=\lim_{m\to\infty}\int_{[0,1]}f\,d\mu_m
=M_n^{-1}.
\]

Every $x\in K_*$ belongs to a unique nested chain of cylinders $(I_n(x))_n$. Their
diameters $4r_n$ tend to zero, so every neighborhood of $x$ contains $I_n(x)$ for all
sufficiently large $n$, and this cylinder has positive $\rho$-mass. Consequently
\begin{equation}\label{eq:full-support}
\supp\rho=K_*.
\end{equation}

\smallskip
\noindent\emph{Step 2: the critical box and packing estimates.}
Let $\mathcal C_n$ be the family of level-$n$ cylinders, namely the $M_n$ closed
intervals whose union is $K_n$. Thus
\[
\#\mathcal C_n=M_n,
\qquad
\operatorname{diam}I=\delta_n:=4r_n,
\qquad
\rho(I)=M_n^{-1}
\quad\text{for every }I\in\mathcal C_n.
\]
The last identity is what we mean by the cylinder-mass formula. The number of cylinders
will control the upper box dimension, while their masses will control the packing
premeasures.

By \eqref{eq:parameter-conditions}, $\delta_n<\delta_{n-1}/5$, so
$(\delta_n)_n$ decreases to zero. Hence, for every sufficiently small $r>0$, there is a
unique $n$ such that
\[
\delta_n\le r<\delta_{n-1}.
\]
Each $I\in\mathcal C_n$ is contained in the interval with the same center and radius
$r$, because the radius of $I$ is $\delta_n/2\le r$. Since
$K_*\subseteq K_n=\bigcup_{I\in\mathcal C_n}I$, these $M_n$ enlarged intervals cover
$K_*$. Therefore
\[
N(K_*,r)\le M_n.
\]
Here
\[
\log M_n=n\log2+O(1)
\]
while
\[
\begin{aligned}
\log\log(1/\delta_{n-1})
&=\log\bigl(L_{n-1}-\log(4A_n)\bigr)\\
&=(n-1)\log2-4\log(n+n_0)+O(1).
\end{aligned}
\]
Since $r<\delta_{n-1}$, we also have
$\log\log(1/r)\ge\log\log(1/\delta_{n-1})$ for all sufficiently small $r$. Consequently,
\[
\frac{\log N(K_*,r)}{\log\log(1/r)}
\le
\frac{\log M_n}{\log\log(1/\delta_{n-1})}
\xrightarrow[n\to\infty]{}1,
\]
which gives
\[
\overdimBlog K_*\le1.
\]
Proposition~\ref{prop:logbox-lower} and \eqref{eq:full-support} give the reverse
inequality, so
\begin{equation}\label{eq:carrier-logbox}
\overdimBlog K_*=1.
\end{equation}

We now use the cylinder-mass formula. For $x\in K_*$, let $I_n(x)$ denote the unique
level-$n$ cylinder containing $x$. If $\delta_n\le r<\delta_{n-1}$, then
$I_n(x)\subseteq B(x,r)$: indeed, $x\in I_n(x)$ and every two points of $I_n(x)$ are
at distance at most $\operatorname{diam}I_n(x)=\delta_n\le r$. Therefore
\begin{equation}\label{eq:mass-lower}
\rho(B(x,r))\ge M_n^{-1}=B^{-1}2^{-n}.
\end{equation}
For $s>1$, monotonicity of $h_s$ gives
\[
h_s(r)\le h_s(\delta_{n-1})
\asymp L_{n-1}^{-s}
\asymp B^{-s}2^{-s(n-1)}(n+n_0)^{4s}.
\]
Here the first comparison follows from
$\log(1/\delta_{n-1})=L_{n-1}-\log(4A_n)\sim L_{n-1}$, and the second follows from
the definition of $L_{n-1}$. Combining this bound with \eqref{eq:mass-lower}, there
is a constant $c'_s>0$ such that, uniformly for $x\in K_*$ and
$\delta_n\le r<\delta_{n-1}$,
\[
\frac{\rho(B(x,r))}{h_s(r)}
\ge c'_s
B^{s-1}2^{(s-1)n}(n+n_0)^{-4s}
\xrightarrow[n\to\infty]{}\infty.
\]
The limit holds because $s>1$, so the exponential factor dominates the polynomial
factor. Moreover,
\[
\frac{h_s(2r)}{h_s(r)}\xrightarrow[r\searrow0]{}1.
\]
Hence, for every $s>1$, there is a constant $c_s>0$ such that
\begin{equation}\label{eq:frostman-reverse}
\rho(B(x,r))\ge c_sh_s(2r)
\end{equation}
for all $x\in K_*$ and all sufficiently small $r$.

Let $\delta>0$ be sufficiently small, and let $\{B(x_j,r_j)\}_j$ be a
$\delta$-packing of $K_*$. Since $2r_j\le\delta$, equation
\eqref{eq:frostman-reverse} applies to every packing ball and gives
\[
\sum_jh_s(2r_j)
\le c_s^{-1}\sum_j\rho(B(x_j,r_j))
\le c_s^{-1}.
\]
The last inequality holds because the packing balls are disjoint and $\rho$ is a
probability measure. Taking the supremum over all $\delta$-packings gives
\[
\mathcal P^{h_s}_\delta(K_*)\le c_s^{-1}.
\]
Letting $\delta$ tend to zero shows that the packing premeasure
$\mathcal P^{h_s}_0(K_*)$ is finite.
If $t>s$, then, for every sufficiently small $\delta$,
\[
\omega(\delta)
:=\sup_{0<u\le\delta}\frac{h_t(u)}{h_s(u)}
=\sup_{0<u\le\delta}\bigl(\log(1/u)\bigr)^{-(t-s)}
\xrightarrow[\delta\searrow0]{}0.
\]
For every $\delta$-packing as above, apply this ratio term by term:
\[
\begin{aligned}
\sum_jh_t(2r_j)
&=\sum_j\frac{h_t(2r_j)}{h_s(2r_j)}h_s(2r_j)\\
&\le\omega(\delta)\sum_jh_s(2r_j)
\le c_s^{-1}\omega(\delta).
\end{aligned}
\]
Taking the supremum over all $\delta$-packings yields
\[
\mathcal P_\delta^{h_t}(K_*)\le c_s^{-1}\omega(\delta).
\]
Letting $\delta$ tend to zero gives $\mathcal P^{h_t}_0(K_*)=0$. Using the single-set
cover of $K_*$ in the definition of packing measure, we obtain
\[
\mathcal P^{h_t}(K_*)\le\mathcal P^{h_t}_0(K_*)=0.
\]
Since for every $t>1$ we may choose $s$ with $1<s<t$,
\[
\dimPlog K_*\le1.
\]
Corollary~\ref{cor:logpacking-lower} supplies the reverse inequality:
\begin{equation}\label{eq:carrier-packing}
\dimPlog K_*=1.
\end{equation}

\smallskip
\noindent\emph{Step 3: extraction inside an arbitrary Hausdorff gauge.}
The set $K_*$ has the desired critical dimensions, but it need not be $h$-null. We now
adapt the Baire-category mechanism in the classical proof of the Iva\v{s}ev-Musatov
theorem; compare \cite[Section~3]{Lyons}. The point is to carry out that argument inside
$K_*$, so that the sharp logarithmic dimensions obtained in Step~2 are not lost. Let
\[
\mathcal R(K_*)
=\{\fhat\nu:\nu\text{ is a Rajchman probability with }\supp\nu\subseteq K_*\}
\subseteq C_0(\R),
\]
equipped with the uniform norm. This space contains $\fhat\rho$ and is complete. Indeed,
suppose that $\fhat{\nu_n}$ converges uniformly to $F\in C_0(\R)$. Compactness of the
probabilities on $K_*$ gives a subsequence converging weakly to a probability $\nu$ with
$\supp\nu\subseteq K_*$. For every $\xi\in\R$, weak convergence gives
\[
\fhat\nu(\xi)=\lim_{k\to\infty}\fhat{\nu_{n_k}}(\xi)=F(\xi).
\]
Since $F\in C_0(\R)$, the measure $\nu$ is Rajchman, and hence
$F\in\mathcal R(K_*)$.

The key density observation is the following. If $V\subseteq K_*$ is relatively open
and dense and $q\ge2$, put
\[
\mathcal U(V,q)
=\{\fhat\nu\in\mathcal R(K_*):\nu(V)>1-1/q\}.
\]
Thus $\fhat\nu\in\mathcal U(V,q)$ precisely when $\nu(K_*\setminus V)<1/q$. This set is
open and dense in $\mathcal R(K_*)$. We verify the two assertions separately.

First, suppose that $\fhat{\nu_k}$ converges uniformly to $\fhat\nu$, with all the
transforms in $\mathcal R(K_*)$. Then, for every $\xi\in\R$,
\[
\int_{K_*}e^{-2\pi i\xi x}\,d\nu_k(x)
\xrightarrow[k\to\infty]{}
\int_{K_*}e^{-2\pi i\xi x}\,d\nu(x).
\]
By the Stone--Weierstrass theorem, finite linear combinations of the characters
$x\mapsto e^{-2\pi i\xi x}$, $\xi\in\R$, are uniformly dense in $C(K_*)$. If $f\in
C(K_*)$ and $P$ is such a linear combination, then
\[
\begin{aligned}
\abs{\int_{K_*}f\,d\nu_k-\int_{K_*}f\,d\nu}
&\le2\lVert f-P\rVert_\infty\\
&\quad+\abs{\int_{K_*}P\,d\nu_k-\int_{K_*}P\,d\nu}.
\end{aligned}
\]
The last term tends to zero, and $P$ can approximate $f$ arbitrarily well. Thus
$\nu_k$ converges weakly to $\nu$; here weak convergence means convergence of the
integrals of every function in $C(K_*)$.

The map $\nu\mapsto\nu(V)$ is lower semicontinuous for weak convergence. Explicitly,
this means that
\begin{equation}\label{eq:open-set-lsc}
\nu(V)\le\liminf_{k\to\infty}\nu_k(V)
\end{equation}
whenever $\nu_k$ converges weakly to $\nu$. To see this, suppose first that
$V\ne K_*$ and define
\[
f_\ell(x)
=\min\{1,\ell\,\dist(x,K_*\setminus V)\},
\qquad x\in K_*.
\]
The continuous functions $f_\ell$ increase pointwise to $\mathbf 1_V$. Hence, by
monotone convergence,
\[
\nu(V)
=\sup_{\ell\ge1}\int_{K_*}f_\ell\,d\nu
\]
while weak convergence and $0\le f_\ell\le\mathbf 1_V$ give, for every fixed $\ell$,
\[
\int_{K_*}f_\ell\,d\nu
=\lim_{k\to\infty}\int_{K_*}f_\ell\,d\nu_k
\le\liminf_{k\to\infty}\nu_k(V).
\]
Taking the supremum over $\ell$ proves \eqref{eq:open-set-lsc}. If $V=K_*$, then
$\mathcal U(V,q)=\mathcal R(K_*)$, so openness is immediate. Suppose now that
$V\ne K_*$ and $\nu(V)>1-1/q$. Monotone convergence allows us to choose $\ell$ such that
\[
\int_{K_*}f_\ell\,d\nu>1-1/q.
\]
The map $\eta\mapsto\int_{K_*}f_\ell\,d\eta$ is continuous for weak convergence, so
the same strict inequality holds throughout a weak neighborhood of $\nu$. Since
$f_\ell\le\mathbf 1_V$, every measure $\eta$ in that neighborhood satisfies
$\eta(V)>1-1/q$. Thus lower semicontinuity means precisely that the strict superlevel
sets of $\nu\mapsto\nu(V)$ are weakly open. Uniform convergence of the transforms
implies weak convergence, as proved above, so $\mathcal U(V,q)$ is open in
$\mathcal R(K_*)$.

We now prove density. Fix $x\in K_*$ and, for $j\ge1$, let
\[
U_{x,j}=V\cap(x-1/j,x+1/j).
\]
Since $V$ is relatively open and dense in $K_*$, the set $U_{x,j}$ is a nonempty
relatively open subset of $K_*$. The equality $\supp\rho=K_*$ therefore implies
$\rho(U_{x,j})>0$. Define
\[
\rho_{x,j}=\frac{\rho|_{U_{x,j}}}{\rho(U_{x,j})}.
\]
By Lemma~\ref{lem:band}, $\rho_{x,j}$ is a Rajchman probability measure, and it gives
full mass to $V$, meaning that $\rho_{x,j}(V)=1$. Moreover, for every $f\in C(K_*)$,
\[
\abs{\int_{K_*}f\,d\rho_{x,j}-f(x)}
\le
\sup_{\substack{y\in K_*\\ \abs{y-x}<1/j}}
\abs{f(y)-f(x)}
\xrightarrow[j\to\infty]{}0.
\]
Thus $\rho_{x,j}$ converges weakly to $\delta_x$.

Consider a finitely supported probability measure
\[
\pi=\sum_{\ell=1}^{L}a_\ell\delta_{x_\ell},
\qquad
a_\ell\ge0,
\qquad
\sum_{\ell=1}^{L}a_\ell=1,
\]
and define
\[
\pi_j=\sum_{\ell=1}^{L}a_\ell\rho_{x_\ell,j}.
\]
Then $\pi_j(V)=1$ and $\pi_j$ converges weakly to $\pi$. By the definition of the
Fourier--Stieltjes transform
\cite[Chapter~VI, equation~(2.3), p.~144]{Katznelson},
for each fixed $j$ we have
\[
\fhat{\pi_j}(\xi)
=\sum_{\ell=1}^{L}a_\ell\fhat{\rho_{x_\ell,j}}(\xi)
\xrightarrow[\abs\xi\to\infty]{}0.
\]
The sum is finite and each term tends to zero, so $\pi_j$ is Rajchman. This also
explicitly proves that every finite convex combination of Rajchman probability measures
is Rajchman.

Finitely supported probability measures are weakly dense in the set of all probability
measures on $K_*$. Indeed, partition $K_*$ into finitely many Borel sets of diameter at
most $\delta$, choose a point in each nonempty part, and move the mass of that part to
the chosen point. Uniform continuity shows that the integrals of any fixed function in
$C(K_*)$ change by a quantity tending to zero with $\delta$. Thus, given a weak
neighborhood of a probability measure on $K_*$, first choose a finitely supported
probability $\pi$ in that neighborhood and then take $j$ sufficiently large that
$\pi_j$ remains in the same neighborhood. Since $\pi_j(V)=1$, this proves that
Rajchman probability measures $\eta$ satisfying $\eta(V)=1$ are weakly dense in all
probability measures on $K_*$. 

It remains to upgrade this measure-weak density to the uniform-norm density needed for
Baire's theorem. Let
\[
\mathcal A(V)
=\{\fhat\eta\in\mathcal R(K_*):\eta(V)=1\}.
\]
The set $\mathcal A(V)$ is dense in $\mathcal R(K_*)$ for the relative weak topology
inherited from $C_0(\R)$. Indeed, if $\eta_k$ converges weakly to $\eta$ as measures on
$K_*$ and $\Lambda\in C_0(\R)^*=M(\R)$, where the identification follows from
the Riesz representation theorem, then Fubini gives
\[
\int_{\R}\fhat{\eta_k}(\xi)\,d\Lambda(\xi)
=\int_{K_*}
\left(\int_{\R}e^{-2\pi i\xi x}\,d\Lambda(\xi)\right)d\eta_k(x).
\]
The function in parentheses is continuous in $x$ by dominated convergence: $\Lambda$
is finite, the integrand has modulus one, and it converges pointwise when $x$ varies.
Hence the right-hand side converges to the same expression with $\eta$ in place of
$\eta_k$. This is precisely weak convergence of the transforms in $C_0(\R)$. The set
$\mathcal A(V)$ is convex by the calculation for finite convex combinations above.
The Hahn--Banach separation theorem implies that the weak and norm closures of a convex
subset of a Banach space coincide. Thus the weak density just proved shows that
$\mathcal A(V)$ is norm dense in $\mathcal R(K_*)$. Since
$\mathcal A(V)\subseteq\mathcal U(V,q)$, the latter set is also dense, as claimed.

Choose a dense sequence $(x_j)_{j\ge1}$ in $K_*$. For every $m$, choose open intervals
$I_{m,j}\ni x_j$ such that
\[
\operatorname{diam}I_{m,j}<\frac1m,
\qquad
\sum_{j=1}^{\infty}h(\operatorname{diam}I_{m,j})<2^{-m}.
\]
This is possible because $h(r)$ tends to zero as $r$ tends to zero, so the intervals can
be chosen to satisfy both requirements, with
$h(\operatorname{diam}I_{m,j})<2^{-m-j}$. Set
\[
V_m=K_*\cap\bigcup_{j=1}^{\infty}I_{m,j},
\qquad
E=\bigcap_{m=1}^{\infty}V_m.
\]
Each $V_m$ is relatively open and dense in $K_*$. For every $m$, the intervals
$(I_{m,j})_{j\ge1}$ cover $E$, have diameter less than $1/m$, and have total $h$-cost
less than $2^{-m}$. Hence
\[
\mathcal H^h_{1/m}(E)\le2^{-m},
\]
and therefore $\mathcal H^h(E)=0$.
Since $\mathcal R(K_*)$ is a nonempty complete metric space and every
$\mathcal U(V_m,q)$ is open and dense, the Baire category theorem shows that the
intersection
\[
\bigcap_{m\ge1}\bigcap_{q\ge2}\mathcal U(V_m,q)
\]
is nonempty. Let $\nu$ be the Rajchman probability corresponding to one of its elements.
For each fixed $m$, membership in $\mathcal U(V_m,q)$ for every $q\ge2$ gives
$\nu(V_m)>1-1/q$ for every $q$, and hence $\nu(V_m)=1$. Since the countably many sets
$V_m$ all have full $\nu$-measure, $\nu(E)=1$.

The set $E$ need not be compact. Since $E$ is Borel and $\nu(E)=1$, inner regularity
\cite[Theorem~2.18, p.~48]{Rudin} gives a compact set $F\subseteq E$ with $\nu(F)>0$.
Lemma~\ref{lem:band} shows that
\[
\mu=\frac{\nu|_F}{\nu(F)}
\]
is Rajchman. Put $K=\supp\mu$. Since $F$ is compact and $\mu(F)=1$,
\[
K\subseteq F\subseteq E\subseteq K_*.
\]
Consequently, monotonicity gives
\[
\mathcal H^h(K)=0,\qquad
\overdimBlog K\le1,\qquad
\dimPlog K\le1.
\]
Since $\supp\mu=K$,
Proposition~\ref{prop:logbox-lower} and Corollary~\ref{cor:logpacking-lower} give equality
in the last two estimates.

For the final assertion, apply the result just proved to the gauge which, for small
$r$, is given by
\[
g(r)=\frac{1}{\log\log(e^e/r)}
\]
and extend it monotonically away from zero. For every $s>0$,
\[
\frac{h_s(r)}{g(r)}
=\frac{\log\log(e^e/r)}{(\log(1/r))^s}
\xrightarrow[r\searrow0]{}0.
\]
Thus $\mathcal H^g(K)=0$ implies $\mathcal H^{h_s}(K)=0$ for every $s>0$, which is
$\dimHlog K=0$.
\end{proof}

For the set $K$ in the final assertion of Theorem~\ref{thm:thin-support},
\[
\dimHlog K=0<1,
\qquad
\dimPlog K=1<2.
\]
Thus neither hypothesis of the Shmerkin--Yavicoli theorem
(Theorem~\ref{thm:shmerkin-yavicoli}) applies to $K$, whereas our main result,
Theorem~\ref{thm:main}, does. The packing value $1$ is optimal among compact sets
supporting Rajchman measures by Corollary~\ref{cor:logpacking-lower}.

\section*{Acknowledgments}
A.~I. and A.~Y. thank \'Akos Magyar, the Erd\H{o}s Center, and the R\'enyi Institute for
their hospitality when this paper was written.

The authors used ChatGPT to assist with the organization and exposition of the manuscript. All mathematical statements, proofs, references, and final text were reviewed and verified by the authors, who take full responsibility for the content.



\begin{thebibliography}{99}

\bibitem{Bluhm}
C.~E.~Bluhm, \emph{Liouville numbers, Rajchman measures, and small Cantor sets},
Proc. Amer. Math. Soc. \textbf{128} (2000), no.~9, 2637--2640.

\bibitem{Bourgain}
J.~Bourgain, \emph{Construction of sets of positive measure not containing an
affine image of a given infinite structure},
Israel J. Math. \textbf{60} (1987), no.~3, 333--344.

\bibitem{BradfordKohutMooroogen}
L.~Bradford, H.~Kohut, and Y.~Mooroogen,
\emph{Large subsets of Euclidean space avoiding infinite arithmetic progressions},
Proc. Amer. Math. Soc. \textbf{151} (2023), no.~8, 3535--3545.

\bibitem{CruzLaiPramanik}
A.~D.~Cruz, C.-K.~Lai, and M.~Pramanik,
\emph{Large sets avoiding affine copies of infinite sequences},
Real Anal. Exchange \textbf{48} (2023), no.~2, 251--270.

\bibitem{Eigen}
S.~J.~Eigen, \emph{Putting convergent sequences into measurable sets},
Studia Sci. Math. Hungar. \textbf{20} (1985), no.~1--4, 411--412.

\bibitem{Erdos}
P.~Erd\H{o}s, \emph{Problems}, Math. Balkanica \textbf{4} (1974), 203--204.

\bibitem{Falconer}
K.~J.~Falconer, \emph{On a problem of Erd\H{o}s on sequences and measurable sets},
Proc. Amer. Math. Soc. \textbf{90} (1984), no.~1, 77--78.

\bibitem{FengLaiXiong}
D.-J.~Feng, C.-K.~Lai, and Y.~Xiong,
\emph{Erd\H{o}s similarity problem via bi-Lipschitz embedding},
Int. Math. Res. Not. IMRN \textbf{2024}, no.~17, 12327--12342.

\bibitem{GaoMooroogenYip}
X.~Gao, Y.~Mooroogen, and C.~H.~Yip,
\emph{On an Erd\H{o}s similarity problem in the large},
Bull. Lond. Math. Soc. \textbf{57} (2025), no.~6, 1801--1818.

\bibitem{Hoeffding}
W.~Hoeffding, \emph{Probability inequalities for sums of bounded random variables},
J. Amer. Statist. Assoc. \textbf{58} (1963), no.~301, 13--30.

\bibitem{IosevichYavicoliFalconer}
A.~Iosevich and A.~Yavicoli,
\emph{Falconer lattice sets and the Erd\H{o}s similarity problem},
arXiv:2604.01493, 2026.

\bibitem{IvashevMusatov}
O.~S.~Iva\v{s}ev-Musatov, \emph{$M$-sets and Hausdorff measure},
Dokl. Akad. Nauk SSSR \textbf{142} (1962), no.~5, 1001--1004 (Russian).

\bibitem{JLM}
Y.~Jung, C.-K.~Lai, and Y.~Mooroogen,
\emph{Fifty years of the Erd\H{o}s similarity conjecture},
Res. Math. Sci. \textbf{12} (2025), Paper No.~9; correction, ibid., Paper No.~50.

\bibitem{Katznelson}
Y.~Katznelson, \emph{An introduction to harmonic analysis}, 3rd ed.,
Cambridge Mathematical Library, Cambridge University Press, Cambridge, 2004.

\bibitem{Kolountzakis}
M.~N.~Kolountzakis, \emph{Infinite patterns that can be avoided by measure},
Bull. London Math. Soc. \textbf{29} (1997), no.~4, 415--424.

\bibitem{KolountzakisCantor}
M.~N.~Kolountzakis, \emph{Sets of full measure avoiding Cantor sets},
Bull. Hellenic Math. Soc. \textbf{67} (2023), 1--11; arXiv:2209.10823.

\bibitem{KolountzakisPapageorgiou}
M.~N.~Kolountzakis and E.~Papageorgiou,
\emph{Large sets containing no copies of a given infinite sequence},
Anal. PDE \textbf{18} (2025), no.~1, 93--108.

\bibitem{KornerIvashev}
T.~W.~K\"orner, \emph{On the theorem of Iva\v{s}ev-Musatov III},
Proc. London Math. Soc. (3) \textbf{53} (1986), no.~1, 143--192.

\bibitem{LiSahlsten}
J.~Li and T.~Sahlsten, \emph{Trigonometric series and self-similar sets},
J. Eur. Math. Soc. \textbf{24} (2022), no.~1, 341--368.

\bibitem{Lyons}
R.~Lyons, \emph{Seventy years of Rajchman measures},
J. Fourier Anal. Appl. Special Issue (1995), 363--377.

\bibitem{MoraEtAlSumsets}
N.~Mora Cu\'ellar, A.~Iosevich, N.~Kulkarni, I.~Rojas Aravena, and A.~Yavicoli,
\emph{The Erd\H{o}s similarity conjecture for two-fold sumsets with a geometric summand},
arXiv:2607.03584, 2026.

\bibitem{Rudin}
W.~Rudin, \emph{Real and complex analysis}, 3rd ed.,
McGraw--Hill Book Co., New York, 1987.

\bibitem{Salem}
R.~Salem, \emph{Algebraic numbers and Fourier analysis},
D.~C. Heath and Company, Boston, MA, 1963.

\bibitem{Schmidt}
W.~M.~Schmidt, \emph{Diophantine approximation},
Lecture Notes in Mathematics, vol.~785, Springer-Verlag, Berlin, 1980.

\bibitem{ShmerkinYavicoli}
P.~Shmerkin and A.~Yavicoli, \emph{Full measure universality for Cantor sets},
Adv. Math. \textbf{495} (2026), Paper No.~110978.

\bibitem{Svetic}
R.~E.~Svetic, \emph{The Erd\H{o}s similarity problem: a survey},
Real Anal. Exchange \textbf{26} (2000/01), no.~2, 525--539.

\bibitem{Wiener}
N.~Wiener, \emph{The Fourier integral and certain of its applications},
Cambridge University Press, Cambridge, 1933.

\end{thebibliography}
\end{document}